\documentclass[12pt]{amsart}
\usepackage{fullpage,amssymb}
\newtheorem{theorem}{Theorem}[section]
\newtheorem{lemma}[theorem]{Lemma}
\newtheorem{corollary}[theorem]{Corollary}
\newtheorem{proposition}[theorem]{Proposition}
\newtheorem*{theorem*}{Proposition~\ref{inv.null}}
\newtheorem*{p*}{Proposition~\ref{hsop}}
\theoremstyle{definition}
\newtheorem{definition}[theorem]{Definition}
\newtheorem{example}[theorem]{Example}
\newtheorem{remark}[theorem]{Remark}

\newcommand{\C}{{\mathbb C}}
\newcommand{\N}{{\mathbb N}}

\newcommand{\SL}{{\operatorname{SL}_n \times \operatorname{SL}_n}}
\newcommand{\B}{{$\beta(R(n,m))$}}

\usepackage{amsmath}
\usepackage{comment}
\usepackage{tikz}

\title{Hilbert series and degree bounds for matrix (semi-)invariants }

\author{Visu Makam}

\thanks{The author was partially supported by NSF grant DMS-1361789}

\begin{document}

\begin{abstract}
We study the ring $R(n,m)$ of invariants for the left-right action of $\SL$ on $m$-tuples of $n \times n$ complex matrices. We show that $R(3,m)$ is generated by invariants of degree $\leq 309$ for all $m$. Then, we use a combinatorial description of the invariants to show that $R(n,m)$ cannot be generated by invariants of degree $<n^2$ for large $m$. We also compute the Hilbert series for several cases. 
\end{abstract}

\maketitle

\section{Introduction}\label{sec1}
We fix our ground field to be $\C$, the field of complex numbers.
\subsection{The ring $R(n,m)$}
Let $\operatorname{Mat}_{n,n}$ be the space of $n \times n$ matrices. We consider the action of $\SL$  on $m$-tuples of matrices, i.e, $\operatorname{Mat}_{n,n}^{m}$ given as follows: 

For $(A,B) \in \SL$, and $(X_1,X_2,\dots ,X_m) \in \operatorname{Mat}_{n,n}^m$, we have
$$ (A,B) \circ (X_1,X_2,\dots, X_m) = (AX_1B^{-1},AX_2B^{-1}, \dots ,AX_mB^{-1}). $$

 We are interested in the ring of polynomial invariants $\C[\operatorname{Mat}_{n,n}^{m}]^{\SL}$, which we denote by $R(n,m)$. This is a graded subring of the polynomial ring $\C[\operatorname{Mat}_{n,n}^{m}]$.

\subsection{The quiver perspective} \label{det.desc}

The $m$-Kronecker quiver is the quiver with 2 vertices $x$ and $y$ with $m$-arrows from $x$ to $y$ :

\begin{center}
\begin{tikzpicture}

\path node (x) at (0,0) []{x};
\path node (y) at (3,0) []{y};

\draw [->,thick] (x) .. controls (1,0.5) and (2,0.5) ..
node[midway,above] {$a_1$}
(y);

\draw [->,thick] (x) .. controls (1,-0.5) and (2,-0.5) ..
node[midway,below] {$a_m$}
(y);

\path node at (1.5,0.2) [] {.};
\path node at (1.5,0) [] {.};
\path node at (1.5,-0.2) [] {.};

\node at (3.5,0) []{$.$};
\end{tikzpicture}
\vspace{5pt}
\end{center}

The ring $R(n,m)$ is the ring of  semi-invariants for the $m$-Kronecker quiver for the dimension vector $(n,n)$. Semi-invariants for quivers have been studied in \cite{DW},\cite{DZ},\cite{Sch1}, and \cite{SVd}, where they exhibit a $\C$-linear spanning set of the semi-invariants. 

Unraveling the theory for our situation, we get a determinantal description of the semi-invariants by the following construction. Let $\{t_{i,j}^p | 1 \leq i,j \leq k, 1\leq p \leq m \}$ be a set of $k^2m$ indeterminates. We describe a $kn\times kn$ matrix $\chi_k$ by splitting it into blocks of size $n \times n$ and describing each block individually. The $(i,j)^{th}$ block is $\sum_{p=1}^m t_{i,j}^p X_p$.  
 
For example, $$\chi_1 = \sum_{p=1}^m t_{1,1}^p X_p $$ 

and

 \renewcommand{\arraystretch}{2}
$$\newcommand*{\temp}{\multicolumn{1}{r|}{}}
\chi_2=\left[\begin{array}{ccc}
\sum\limits_{p=1}^m t^p_{1,1} X_p  &\temp & \sum\limits_{p=1}^m t^p_{1,2} X_p \\ \cline{1-3}
\sum\limits_{p=1}^m t^p_{2,1} X_p &\temp & \sum\limits_{p=1}^m t^p_{2,2} X_p\\
\end{array}\right] .
$$\\

 View $\det(\chi_k)$ as a polynomial in the indeterminates $t_{i,j}^p$. The coefficients are expressions in the entries of the matrices. Further these coefficients are homogeneous polynomials (of degree $kn$) in the entries of the matrices, and they are invariant under the action of $\SL$. 
 
We get the following : 
  
\begin{enumerate}
\item There are semi-invariants in degree $d$ only if $n|d$.
\item  The coefficients of $\det(\chi_k)$ form a $\C$-linear spanning set for the invariants of degree $kn$.

\end{enumerate}

\begin{remark}
The spanning set described above is usually not a basis, and it is difficult to extract a basis. 
\end{remark}

\subsection{Degree bounds and Hilbert series}

\begin{definition}
For a rational finite dimensional representation $V$ of a linearly reductive group $G$, we define $\beta(\C[V]^G)$ to be the smallest integer $N$ such that the invariants of degree~$\leq~N$ generate $\C[V]^G$. 
\end{definition}

Computing $\beta(\C[V]^G)$ is a hard problem in general. There is a general method to get upper bounds, but it is far from being satisfactory in our case. See \cite[Theorem 4.7.4, Proposition 4.7.12, Proposition 4.7.14]{DK}.
 
A theorem of Weyl (see \cite[Section~7.1,Theorem~A]{KP}) tells us that $\beta(R(n,m)) \leq \beta (R(n,n^2))$. So we make the following definition.

\begin{definition}
 $\beta_U(n) := \beta(R(n,n^2))$ is called the universal bound.
 \end{definition}

In some small cases, tight upper bounds have been computed. 

\begin{itemize}
\setlength\itemsep{10pt}
\item $\mathbf{m=1:}$ It is easy to see that $R(n,1) = \C[\det(X_1)]$, i.e a polynomial ring in one variable.  

\item $\mathbf{m=2:}$ The coefficients of $\det(t_1 X_1 + t_2 X_2)$ span the invariants of degree $n$. It turns out that these $n+1$ coefficients are algebraically independent and generate $R(n,2)$. Hence $R(n,2)$ is in fact a polynomial ring! These results can be found in \cite{Hap}, \cite{Hap2}, and \cite{Koike}. \\
\end{itemize}
The above two situations deal with finite and tame quivers as the $1$-Kronecker quiver is finite and the $2$-Kronecker quiver is tame. For $m \geq 3$, the situation gets considerably more complicated. The rings are not polynomial rings anymore. 

Before we discuss more previous results, we discuss a few general definitions and results. Details can be found in \cite{DK}. The ring $R(n,m)$ is a finitely generated graded algebra, and hence has a homogeneous system of parameters (hsop), i.e a set of algebraically independent homogeneous elements $f_1,\dots, f_d$ such that $R(n,m)$ is a finitely generated module over $\C[f_1,\dots, f_d]$. Further the Hochster-Roberts theorem (see \cite[Theorem~2.5.5]{DK},\cite{HR}) tells us that $R(n,m)$ is Cohen-Macaulay, and hence is a free module over any hsop. The free module generators can be chosen to be homogeneous. The elements of the hsop are called primary invariants and the free module generators are called secondary invariants. In particular, this gives a generating set for $R(n,m)$ and hence an upper bound for \B. \\

\begin{definition} {\textbf{Hilbert series:}}
The Hilbert series of a graded $\C$-algebra $R = \oplus_{d \in \N} R_d$ is defined as $$H(R,t) = \sum_d \dim(R_d)t^d.$$ 
\end{definition}

In the above definition, $\N$ denotes the set of non-negative integers. Observe that a set of primary and secondary invariants for $R(n,m)$ suffice to compute its Hilbert series, and write it as a rational function.

\begin{definition} {\textbf{Multi-graded Hilbert series:}}
If a $\C$-algebra $R$ has a multi-grading, i.e. $R = \oplus_{d \in \N^l} R_d$, then we define its multi-graded Hilbert series $$\mathbb{H}(R,t_1,t_2,\dots,t_l) = \sum_{d = (d_1,d_2,\dots,d_l) \in \N^l} \dim(R_d)t_1^{d_1}t_2^{d_2}\dots t_l^{d_l}.$$  
\end{definition}

\begin{remark}
Any multi-graded algebra can be viewed as a (singly) graded algebra by considering the total degree. Then one can recover the Hilbert series from the multi-graded Hilbert series by specializing all the variables $t_i = t, 1 \leq i \leq l$, i.e,

$$H(R,t) = \mathbb{H}(R,t,\dots, t).$$ 
\end{remark}

\begin{example}
Let $R = \C[x_1,x_2,\dots,x_l]$ be the polynomial ring in $l$ variables with the natural multi-grading. Then $$\mathbb{H}(R,t_1,t_2,\dots,t_l) = \prod\limits_{i=1}^l \frac{1}{1-t_i},$$ and $$H(R,t) =\mathbb{H}(R,t,\dots,t) =  \frac{1}{(1-t)^l}\  .$$ 
\end{example}

The ring $R(n,m)$  is a multi-graded subring of $\C[\operatorname{Mat}_{n,n}^m]$ and hence has an $\N^m$-grading. One can then take total degree to get an $\N$-grading.

\begin{itemize}

\item \textbf{n = 2 :} In \cite{Dom00a}, Domokos computes the multi-graded Hilbert series of $R(2,m)$ by computing it as an integral. See \cite{DK} for a general method to compute multi-graded Hilbert series by computing integrals. He uses this to deduce  $\beta_U (2) \leq 4$, and is able to show that the bound is tight, and hence we get $\beta_U(2) = 4$.  
 
\item \textbf{n = 3, m = 3 :} In \cite{Dom00b}, Domokos computes the multi-graded Hilbert series of $R(3,3)$. This time, he makes use of the fact that $R(n,m) \twoheadrightarrow S(n,m-1)$, where $S(n,m-1)$ is the invariant ring of $(m-1)$-tuples of $n \times n$ matrices under the conjugation action of $\operatorname{GL}_n$. He then uses results of \cite{Teranishi} on $S(3,2)$ to compute the multi-graded Hilbert series for $R(3,3)$, by explicitly finding primary and secondary invariants. He gets

$$\mathbb{H}(R(3,3),t_1,t_2,t_3) = \frac{1 + t_1^3t_2^3t_3^3}{(\prod_{i}(1-t_i)^3)( \prod_{i \neq j} (1-t_i^2t_j)) (1-t_1t_2t_3) (1 - t_1^2t_2^2t_3^2)}. $$ \\

 On specializing, we get the Hilbert series 

$$H(R(3,3),t) = \frac{1 + t^9}{(1-t^3)^{10} (1-t^6)} . $$ 

 From this he deduces that $\beta(R(3,3)) \leq 9$. Further this set of primary and secondary invariants turns out to be a minimal generating set, and he is able to argue that $\beta(R(3,3)) = 9$.  
 \end{itemize}

The best known general bound is $\beta(R(n,m)) \leq \beta_U(n) \leq O(n^4 \cdot ((n+1)!)^2)$ (see \cite{IQS},\cite{IQS2}). At this point, it is unclear whether one would expect polynomial bounds (in $n$ and $m$) for $\beta(R(n,m))$.

\subsection{Null Cone}
We introduce the null cone, which is an important tool in computational invariant theory. A good understanding of the null cone could lead to strong degree bounds. 

\begin{definition}{\textbf{Null Cone :}}
For a rational representation $V$ of a linear reductive group $G$, the null cone $\mathcal{N}_V$ is the zero set of all homogeneous invariant polynomials of positive degree
$$\mathcal{N}_V = \{ v \in V | f(v) = 0 \  for\  all \ f \in \C[V]^G_+ \}.$$ 
\end{definition}

\begin{definition}
$\gamma(\C[V]^G)$ is defined as the smallest integer $D$ such the invariants of degree~$\leq~D$ define the null cone $\mathcal{N}_V$.
\end{definition}

One can use bounds on $\gamma(R(n,m))$ to give bounds on $\beta(R(n,m))$ (see \cite[Theorem~4.7.4]{DK}) and consequently bounds on $\beta_U(n)$ as well. In particular, a polynomial bound on $\gamma(R(n,m))$ would give us polynomial bounds for $\beta(R(n,m))$ and $\beta_U(n)$.

\subsection{Organisation and Results} In Section~\ref{n.c}, we exhibit invariants that define the null cone for $R(3,m)$. We have

\begin{proposition} \label{inv.null}
The null cone for $R(3,m)$ is defined by a finite set of invariants of degree $\leq 6$, namely

\begin{itemize}
\item A set of $\leq 9m-16$ degree $3$ invariants that define the same subvariety as the vanishing of all degree $3$ invariants. 
\item The degree $6$ invariants $g_{i,j,k} := det\left( \arraycolsep=4 pt\def\arraystretch{1.4} \begin{array}{c|c}
X_j & X_i  \\
\hline
X_i & X_k \\
 \end{array} \right)$ for $1 \leq i < j < k \leq m$.  \\

\end{itemize}
\end{proposition}

We can deduce from the proof of Proposition~\ref{inv.null} that the invariants of degree $6$ are necessary to define the null cone if $m \geq 3$. 

\begin{corollary}\label{gamma}
For $m \geq 3$, we have $\gamma(R(3,m)) = 6$.
\end{corollary}

This in turn gives us a hsop.

\begin{proposition} \label{hsop}
For $m \geq 3$, there exists a set of $9m - 16$ invariants of degree $6$ that form a hsop for $R(3,m)$. 
\end{proposition}

We then use Proposition~\ref{inv.null} and Proposition~\ref{hsop} to find upper bounds for $R(3,m)$ for various $m$ in Section~\ref{u.b}. In particular,
 
 \begin{proposition} \label{g.d}
 The ring $R(3,m)$ is generated by invariants of degree $\leq 309$ for all $m$, i.e, $\beta_U(n) \leq 309$.
 \end{proposition}

We introduce a combinatorial description of the invariants in Section~\ref{comb}. This description helps us get a formula to compute the dimensions of the graded pieces of $R(n,m)$.

\begin{lemma} \label{d.f} The dimenson of $R(n,m)_{kn}$ is given by the computable formula
$$\dim (R(n,m)_{kn} )= \sum\limits_{\lambda \vdash kn} a_{k^n,k^n,\lambda} (\dim\  S_{\lambda}(\C^m)).$$
\end{lemma}

Here $S_{\lambda}$ is the Schur functor corresponding to the partition $\lambda$, and $a_{\lambda,\mu,\nu}$  denote Kronecker coefficients, which are known to be invariant under permutation of $\lambda,\mu$ and $\nu$. In this paper, we use several well known results on Schur functors and symmetric functions, and these can be found in standard books such as \cite{Fulton}, \cite{Mac}, and \cite{Stanley}.

We further analyze this combinatorial description in Section~\ref{l.b} to prove that invariants of degree $< n^2$ cannot generate $R(n,m)$ for $m\geq n^2$, i.e, 

\begin{theorem}\label{t}
Suppose $m \geq n^2$, then $\beta(R(n,m)) \geq n^2$. In particular $\beta_U(n) \geq n^2$. 
\end{theorem}

Finally in Section~\ref{hilb}, we use the results in \cite{DW2} to get denominators of low degree for the Hilbert series, which makes computations more feasible, and give explicit computations.

\section{Null cone for $R(3,m)$} \label{n.c}
\subsection{Krull dimension of $R(n,m)$}
There is a formula for the dimension of the ring of semi-invariants of a quiver for a given dimension vector in terms of the canonical decomposition of the dimension vector (see \cite[Proposition 4] {Kac2}). In the case of the $m$-Kronecker quiver, the canonical decomposition of the dimension vector $(n,n)$ is the following :

\begin{itemize}
\item \textbf{$m = 1,2:$} The canonical decomposition is $(n,n) = (1,1)^{\oplus n}$;
\item \textbf {$m\geq 3:$} $(n,n)$ is an imaginary Schur root and its canonical decomposition is trivial, i.e, $(n,n) = (n,n).$
\end{itemize}

The cases for $m=1,2$ have already been dealt with, so we only apply Kac's formula for $m \geq 3$ to get the following lemma.
\begin{lemma} \label{dim ring}
If $m \geq 3$, then we have $\dim R(n,m) = mn^2 - 2(n^2 -1).$ 
\end{lemma}

\subsection{Invariants defining the null cone}
Proposition~\ref{inv.null} gives a finite set of invariants that define the null cone. We rely heavily on the results in \cite{Dom00b} for proving it.
\begin{proof} [Proof of Proposition~\ref{inv.null}]
Let $Z$ denote the vanishing set of all the degree $3$ invariants. Note that the dimension of $R(3,m)$ is $9m -16$, by Lemma~\ref{dim ring}. Hence, there is a set of $\leq 9m-16$  degree $3$ invariants that defines $Z$, by the Noether normalization lemma (see \cite[Lemma~2.4.7]{DK}).

In \cite{Dom00b}, Domokos analyses the maximal singular matrix spaces in order to compute a hsop for $R(3,3)$. We quickly summarize the results which we'll use. 

A singular matrix space is a linear subspace of the space of matrices which does not contain an invertible matrix. The $m$-tuples in $Z$ are precisely the $m$-tuples whose span is a singular matrix space, by the determinantal description in Section~\ref{det.desc}.

In \cite{Dom00b}, Domokos classifies the maximal singular $3 \times 3$ spaces as being equivalent to one of 4 types, which are denoted $\mathcal{H}_i, i= 1,2,3,4$ (see \cite[Proposition~2.2]{Dom00b}).  A triple of matrices belongs to the null cone if and only if it belongs to a maximal singular space of type  $\mathcal{H}_1, \mathcal{H}_2$, or $\mathcal{H}_3$, by \cite[Proposition~3.2]{Dom00b}. The same proof goes through for an $m$-tuple of matrices for any $m \geq 3$. Domokos remarks after \cite[Proposition~2.2]{Dom00b} that any $2$-dimensional singular space is contained in $\mathcal{H}_1,\mathcal{H}_2$, or $\mathcal{H}_3$. $\mathcal{H}_4$ is the space of skew-symmetric matrices, and in particular is a $3$-dimensional space.  In \cite{Dom00b}, Domokos shows that the invariant $\det\left(  \arraycolsep=4pt\def\arraystretch{1.4}\begin{array}{c|c}
X_2 & X_1  \\
\hline
X_1 & X_3 \\
 \end{array} \right)$
does not vanish on a triple of matrices $(X_1,X_2,X_3)$ if the span of the triple is equivalent to $\mathcal{H}_4$.

 Suppose an $m$-tuple $(X_1,X_2,\dots,X_m)$ is in $Z$, but not in the null cone, then the span of the $m$-tuple is equivalent to $\mathcal{H}_4$, and hence $3$-dimensional. Hence, there exist 3 matrices $X_i,X_j,X_k$ which span the space. Hence  $g_{i,j,k}$ is an invariant that does not vanish on the given $m$-tuple. 
\end{proof}
 
\begin{proof}[Proof of Corollary~\ref{gamma}]
By Proposition~\ref{inv.null}, the invariants of degree $\leq 6$ define the null cone. We observe from the proof of Proposition~\ref{inv.null}, that the degree $3$ invariants are not sufficient to define the null cone if $m\geq 3$.
 \end{proof}
  
 \begin{remark}
 The set of invariants in Proposition~\ref{inv.null} forms a hsop for $m = 3$, but not for $m \geq 4$ as the number of invariants is larger than the dimension of the ring. 
 \end{remark} 

\subsection{A hsop for $R(3,m), m\geq 3$}

\begin{proof}[Proof of Proposition~\ref{hsop}]
Recall that $\dim(R(3,m)) = 9m -16$. Since invariants of degree $3$ and degree $6$ define the null cone, it is clear that just the set of  invariants of degree $6$ define the null cone. By the Noether normalization lemma (see \cite[Lemma~2.4.7]{DK}), we conclude that there exists $9m - 16$ degree $6$ invariants that form a hsop.
\end{proof}

\section{Upper bounds for $\beta(R(3,m))$} \label{u.b}

We want to bound the degrees of primary and secondary generators, in order to obtain upper bounds on $\beta(R(3,m))$. The following result of Knop is very useful in that regard.

\begin{theorem} [\cite{Knop1}] \label{Knop1}
Let $V$ be a rational representation of a semisimple connected group $G$. Let $r$ be the Krull dimension of $\C[V]^G$, then 

$$\deg (H(\C[V]^G)) \leq -r.$$

\end{theorem}

In \cite{DK}, this is used to get the following result.

\begin{proposition}[\cite{DK}]\label{abc}
Let $V$ and $G$ be as in the Theorem~\ref{Knop1}. Suppose $f_1,f_2,\dots,f_l \in \C[V]^G$ are homogeneous invariants that define the null cone. Let $d_i = \deg f_i$. Then $$\beta(\C[V]^G) \leq \max \{d_1,d_2,\dots,d_l,d_1 + d_2 + \dots +d_l - l\}.$$

\end{proposition}

There is a stronger result by Knop on the degree of the Hilbert series.
\begin{theorem}[\cite{Knop2}] \label{Knop2}
Let $V$ be a rational representation of a semisimple connected group $G$. Let $Z = \{ v \in V | \dim G_v > 0 \}$. Then $$\deg (H(\C[V]^G)) = - \dim V \Longleftrightarrow \operatorname{codim} (Z) \geq 2.$$

\end{theorem}

In \cite{Dom00a}, the codimension condition was proved for $R(n,m)$ for $m \geq 3$, and $n \geq 2$. Since this stronger result on the degree of the Hilbert series holds, we can get a stronger result by repeating the proof of  Proposition~\ref{abc}  (see the proof of \cite[Corollary~4.7.7]{DK}). Lemma~\ref{dim ring} implies that for $m \geq 3$, the difference between $\dim R(n,m)$ and $\dim \operatorname{Mat}_{n,n}^m$ is $2n^2 - 2$. So, we get

\begin{proposition}\label{deg.bds}
Let $m \geq 3$ and $n\geq 2$. Suppose $f_1,f_2,\dots,f_l \in R(n,m)$ are homogeneous invariants that define the null cone. Let $d_i = \deg f_i$. Then

$$\beta(R(n,m)) \leq \max \{d_1,d_2,\dots,d_l,d_1 + d_2 + \dots +d_l - l - 2n^2 + 2)\}.$$
\end{proposition}

For computing upper bounds for $\beta(R(3,m))$, we can apply Proposition~\ref{deg.bds} to the set of invariants defining the null cone given by either Proposition~\ref{inv.null} or Proposition~\ref{hsop}. For $m \leq 3$, tight upper bounds have already been computed. For $4 \leq m \leq 6$, Proposition~\ref{inv.null} gives better bounds, whereas for $m \geq 7$, Proposition~\ref{hsop} gives better bounds. So we get the following table.  

$$ \arraycolsep= 4 pt\def\arraystretch{1.4}\begin{array}{c|c}
m & \text{Upper bound for }  \beta(R(3,m)) \\
\hline
1 & 3 \\
2 & 3 \\
3 & 9 \\
4 & 44 \\
5 & 92 \\
6 & 160 \\
7 & 219 \\
8 & 264 \\
9 & 309 \\ 
 \end{array} $$
 
\begin{proof}[Proof of Proposition~\ref{g.d}]
As remarked in the introduction, a theorem of Weyl (see \cite[Section~7.1,Theorem~A]{KP}) gives us $\beta(R(3,m)) \leq \beta(R(3,9)) \leq 309$.
\end{proof}

\section{Combinatorial description of $R(n,m)$} \label{comb}
In this section, we introduce a combinatorial description of the invariants. This description has been studied before (see \cite{ANS} and \cite{AS}). We write $\lambda \vdash d$ to denote that $\lambda$ is a partition of $d$. We denote by $S_{\lambda}$, the Schur functor corresponding to the partition $\lambda$. We identify $\operatorname{Mat}_{n,n}$ with $\C^n \otimes \C^n$, and consequently identify $\operatorname{Mat}_{n,n}^{m}$ with $\C^n \otimes \C^n \otimes \C^m$. Thus 
$$\C[\operatorname{Mat}_{n,n}^{m}] = \C[\C^n \otimes \C^n \otimes \C^m]  = \bigoplus\limits_{d = 0}^{\infty} S_d(\C^n \otimes \C^n \otimes \C^m).$$ 

Let $\lambda \vdash d$. We have the decomposition 
\begin{equation*}
S_{\lambda}(V \otimes W) = \bigoplus\limits_{ \mu, \nu \vdash d} (S_{\mu}(V) \otimes S_{\nu}(W))^{a_{\lambda,\mu,\nu}},
\end{equation*}

where $a_{\lambda,\mu,\nu}$ are known as the Kronecker coefficients. A particular case is the Cauchy formula,  

\begin{equation*}
S_d(V \otimes W) = \bigoplus\limits_{\lambda \vdash d} S_{\lambda}(V) \otimes S_{\lambda}(W).
\end{equation*}

 Applying the above two decompositions, we get

\begin{align*}
S_d (V \otimes W \otimes Z) &=  \bigoplus\limits_{\lambda \vdash d} S_{\lambda}(V \otimes W) \otimes S_{\lambda}(Z) \\ 
&= \bigoplus\limits_{\lambda \vdash  d} \left( \bigoplus\limits_{\mu, \nu \vdash d} (S_{\mu}(V) \otimes S_{\nu}(W) )^{a_{\lambda,\mu,\nu}}\right) \otimes S_{\lambda}(Z)\\
& = \bigoplus\limits_{\lambda,\mu,\nu \vdash d} (S_{\mu}(V) \otimes S_{\nu}(W) \otimes S_{\lambda}(Z))^{a_{\lambda,\mu,\nu}}.
\end{align*} 
 
This shows in particular that the Kronecker coefficients are invariant under permutating $\lambda,\mu$ and $\nu$. The above is essentially a decomposition of $S_d (V \otimes W \otimes Z)$ as a direct sum of irreducible representations of $\operatorname{GL}(V) \times \operatorname{GL}(W) \times \operatorname{GL}(Z)$. 

\begin{proposition}\label{Schur:descr}
The invariants of $R(n,m)$ have the following description.
\begin{enumerate} 
\item $R(n,m)_d = 0$ if $n \nmid d$.
\item$R(n,m)_{kn} = S_{k^n}(\C^n) \otimes S_{k^n}(\C^n) \otimes \left(\bigoplus\limits_{\lambda \vdash kn} S_{\lambda}(\C^m)^{a_{k^n,k^n,\lambda}}\right).$ 
\end{enumerate}
\end{proposition}

\begin{proof}

We want the polynomials which are invariant under the action of $\SL$.  $\operatorname{SL}_n$ acts trivially on the irreducible representations of $\operatorname{GL}_n$ corresponding to the rectangular partitions of length $n$ (i.e the powers of the determinant representation). On all other irreducible representations, $\operatorname{SL}_n$ acts with no invariants. 

Thus the $\SL$  invariants  of degree $d$ are the summands $(S_{\mu}(\C^n) \otimes S_{\nu}(\C^n) \otimes S_{\lambda}(\C^m))^{a_{\lambda,\mu,\nu}}$ in the decomposition of $S_d(\C^n \otimes \C^n \otimes \C^m)$ where $\mu, \nu$ are rectangular partitions of length $n$, i.e, $\mu = \nu = k^n$ for some $k$. So, in particular, unless $d$ is a multiple of $n$, we cannot have any invariants. This proves (1). For (2), 

 \begin{align*}
R(n,m)_{kn} &=  \bigoplus\limits_{\lambda \vdash kn} (S_{k^n}(\C^n) \otimes S_{k^n}(\C^n) \otimes S_{\lambda}(\C^m))^{a_{k^n,k^n,\lambda}} \\
&= S_{k^n}(\C^n) \otimes S_{k^n}(\C^n) \otimes \left(\bigoplus\limits_{\lambda \vdash kn} S_{\lambda}(\C^m)^{a_{k^n,k^n,\lambda}}\right).
\end{align*}
\end{proof}

\begin{proof} [Proof of Lemma~\ref{d.f}]

Since $S_{k^n}(\C^n)$ is 1-dimensional, as $\operatorname{GL}_m$ representations, we have $$ R(n,m)_{kn} =  \bigoplus\limits_{\lambda \vdash kn} S_{\lambda}(\C^m)^{a_{k^n,k^n,\lambda}}.$$ Hence we get the formula
$$\dim (R(n,m)_{kn} )= \sum_{\lambda \vdash kn} a_{k^n,k^n,\lambda} (\dim\  S_{\lambda}(\C^m)). $$

\end{proof}

\begin{remark}
Let $\lambda,\mu \vdash d$. If $T_{\lambda} \ (resp.\  T_{\mu})$ denotes the irreducible representation of the symmetric group on $d$ letters corresponding to the partition $\lambda\ (resp.\ \mu)$, then $$T_{\lambda} \otimes T_{\mu} = \bigoplus\limits_{\nu} T_{\nu}^{a_{\lambda,\mu,\nu}}.$$

\end{remark}

\begin{example}
$T_{1^n}  \otimes T_{1^n} = T_{n}$. Therefore by Lemma~\ref{d.f}, $$ \dim (R(n,m)_n) = \dim (S_n(\C^m)) =  {m+n-1\choose n}.$$ 
\end{example}

\begin{example}
$T_{2^3}  \otimes T_{2^3} = T_6 + T_{(4,2)} + T_{2^3} + T_{(3,1,1,1)}$. Therefore by Lemma~\ref{d.f},
$$ \dim(R(3,m)_6) = \dim S_{6}(\C^m) + \dim S_{(4,2)}(\C^m) + \dim S_{2^3}(\C^m) + \dim S_{(3,1,1,1)}(\C^m).$$

\end{example}

\section{Lower bounds for $\beta(R(n,m))$} \label{l.b}
Let $R \subset \C[W \otimes V]$ be a $\operatorname{GL}(V)$ stable graded subring. Then each $R_d$ is a finite dimensional $\operatorname{GL}(V)$ representation, and we can decompose it as a direct sum of irreducibles, i.e, $$R_d = \bigoplus\limits_{\lambda \vdash d} (S_{\lambda}(V))^{n_{\lambda}}, n_{\lambda} \in \N.$$ Note here that as $\operatorname{GL}(V)$ representations, the $k^{th}$ exterior power $\bigwedge^k(V)$ is $S_{1^k}(V)$ for all positive integers $k$.

\begin{proposition} \label{qwerty}
Let $R \subset  \C[W \otimes V]$ be a $\operatorname{GL}(V)$ stable graded subring. Assume 
\begin{enumerate} 
\item $\bigwedge^i(V)$ does not occur in the decomposition of $R_i$, for $i = 1,2,...,t-1$ ;
\item  $\bigwedge^t(V)$ occurs in the decomposition of $R_t$ at least once ; 
\item $\dim V \geq t$.
\end{enumerate}
Then $\beta(R) \geq t$.

\end{proposition}

\begin{proof}
We have a $\operatorname{GL}(V)$ equivariant map $R_i \otimes R_{t-i} \rightarrow R_t$ given by multiplication. We can collect the maps for various $i$ to get a map 
 $$\varphi: \bigoplus\limits_{i =1}^{\lfloor t/2 \rfloor} R_i \otimes R_{t-i} \rightarrow R_t.$$ It is clear that if R is generated by invariants of degree $< t$, then $\varphi$ is surjective. 
 
Let  $\lambda \vdash a$ and $\mu \vdash b$. Recall the well known identity $$S_{\lambda}(V) \otimes S_{\mu}(V)~=~\oplus_{\nu} (S_{\nu}(V))^{c_{\lambda,\mu}^{\nu}},$$ where $c_{\lambda,\mu}^{\nu}$ are the Littlewood-Richardson coefficients. By the Littlewood-Richardson rule, if $\bigwedge^{a+b}(V) = S_{1^{a+b}}(V)$ appears in the decomposition for $S_{\lambda}(V) \otimes S_{\mu}(V)$, then $\lambda = 1^a$ and $\mu = 1^b$.
 
We assume that for $1 \leq i  \leq t - 1$, $\bigwedge^i(V) = S_{1^i}(V)$ does not occur in the decomposition for $R_i$. Hence, $\bigwedge^t(V)$ does not occur in the decomposition for $R_i \otimes R_{t-i}$, and hence does not occur in the decomposition for $\bigoplus_{i =1}^{\lfloor t/2 \rfloor} R_i \otimes R_{t-i}$, and thus not in the decomposition of its image under $\varphi$. But in the decomposition of $R_t$, there is at least one copy of $\bigwedge^t(V)$. Since $\dim V \geq t$, we are guaranteed that $\bigwedge^t(V)$ is non-empty. So $\varphi$ cannot be surjective.
 
Thus $R$ cannot be generated in degree $<t$ as the invariants corresponding to the isotypic component for $\bigwedge^t(V)$ cannot be generated by smaller degree invariants.
\end{proof}

\begin{proof}[Proof of Theorem~\ref{t}]
We want to apply Proposition~\ref{qwerty} to the ring $R(n,m)$, via the combinatorial description in Section~\ref{comb}. Take $W = \C^n \otimes \C^n$ and $V = \C^m$. Then by the results in Section~\ref{comb}, we have that

\[ R_i = \begin{cases}
         \bigoplus\limits_{\lambda \vdash kn} (S_{\lambda}(V))^{a_{k^n,k^n,\lambda}} & \text{if}\  i =  kn,\\
          0 & \text{otherwise}.
          \end{cases}\]         

From the representation theory of the symmetric group, we know that $T_{k^n} \otimes T_{1^{kn}} = T_{n^k}$. Moreover, since the Kronecker coefficients are invariant under permutations, we have \[ a_{k^n,k^n,1^{kn}} = \begin{cases} 
      0 & \text{if}\  k \neq n,\\
      1 & \text{if} \ k=n.
   \end{cases}
\]

This gives the first two conditions required for Proposition~\ref{qwerty} (for $t = n^2$) . Since we assume $\dim V = m \geq n^2$, the third condition holds as well. Hence we have $\beta(R(n,m)) \geq n^2$. 
   
   \end{proof}

In fact, we can describe explicitly these invariants in degree $n^2$. For a matrix $M$, denote by $\overline{M}$, a column matrix obtained by stacking the columns of $M$.

\begin{example}
If $M = \left( \arraycolsep=4 pt\def\arraystretch{1.4} \begin{array}{cc}
a& b \\
c & d \\
 \end{array} \right)$, then $\overline{M} = \left( \arraycolsep=4 pt\def\arraystretch{1.4} \begin{array}{c}
a \\
c \\
b \\
d \\
 \end{array} \right).$
\end{example}

Define a function $f$ on $n^2$-tuples of $n\times n$ matrices by $$f(X_1,X_2,\dots,X_{n^2}) = \det \left(\begin{array}{c|c|c|c} \overline{X_1} & \overline{X_2} & \dots & \overline{X_{n^2}} \\ \end{array} \right).$$

Then $f \in R(n,n^2)_{n^2}$ is the unique invariant (upto scalars) in the isotypic component corresponding to $\bigwedge^{n^2}(\C^{n^2})$. For $n = 2$, this is the invariant of degree $4$ constructed in \cite{Dom00a}.

 \begin{remark}
Theorem~\ref{t} gives a tight lower bound for $n \leq 2$, i.e $\beta_U(n) = n^2$ for $n \leq 2$.
\end{remark}

\section{Computing Hilbert series} \label{hilb}

We have seen in Section~\ref{sec1} that for the cases $m =1,2$, $R(n,m)$ is a polynomial ring. It is also clear that $R(1,m)$ is a polynomial ring since $\operatorname{SL}_1$ is trivial. For $R(2,m)$, the Hilbert series has already been computed by Domokos in \cite{Dom00a}. So, we restrict to the cases $m \geq 3$, $n \geq 3$.  Notice that for these cases, we have $\deg H(R(n,m),t) = - \dim \operatorname{Mat}_{n,n}^m$, as discussed in Section~\ref{u.b}.

\begin{remark} \label{den.num}
If we can compute a denominator for the  Hilbert series of $R(n,m)$, then we can compute the polynomial in the numerator once we know the dimensions of the graded pieces of $R(n,m)$ upto the degree of the numerator, which is given by  
 $$\deg(Numerator) = \deg(Denominator) - n^2m.$$
\end{remark}

\begin{remark}\label{comp.diff}
Computing $\dim R(n,m)_{kn}$ is a hard task even with a computer, and is the bottleneck for these computations. So, it is desirable to minimize the degree of the numerator as much as possible, and hence it is desirable to minimize the degree of the denominator. 
\end{remark}

Fortunately, the theory of universal denominators (see \cite{D}, \cite{DW2}) gives us strong results in our case. We first renormalize our grading to agree with the grading in \cite{DW2}.

\begin{definition}
The renormalized Hilbert series is defined as 
$$\widetilde{H}(R(n,m),t) = \sum\limits_{k = 0}^{\infty} \dim R(n,m)_{kn} t^k.$$
\end{definition}

\begin{remark}
The usual Hilbert series and the renormalized Hilbert series are related by $$\widetilde{H}(R(n,m),t^n) = H(R(n,m),t).$$
\end{remark}

The most relevant result is \cite[Corollary~1]{DW2}. We restate it for our situation.

\begin{proposition} [\cite{DW2}] \label{denom}
Let $r$ be the Krull dimension of $R(n,m)$. Then 
$$\widetilde{H}(R(n,m),t) = \frac{P(t)}{(1-t)^r},$$

where $P(t)$ is a polynomial with integer coefficients.
\end{proposition}

This gives us denominators of the lowest degree possible, making several computations accessible. Domokos proved a functional equation for the Hilbert series of $R(n,m)$ for $m\geq 3$, $n\geq 2$ in \cite{Dom00a}. This implies that when we use the universal denominator, the coefficients of the polynomial in the numerator are palindromic, so we need to compute only half the coefficients. In view of Remarks~\ref{den.num}-\ref{comp.diff}, this makes a few more computations feasible.

We give a few explicit computations that we are able to compute. 

\begin{enumerate} 
\setlength\itemsep{1em}

\item $\widetilde{H}(R(3,3),t) = \displaystyle \frac{1-t + t^2}{(1-t)^{11}}$.
 This was already computed by Domokos in \cite{Dom00b}. We remark that even though $\beta(R(3,3)) = 9$, we only needed the $\dim (R(3,3)_3)$ to compute the Hilbert series.

\item $\widetilde{H}(R(3,4),t) = \displaystyle \frac{1 + 20t^2 +  20t^3 +  55t^4 + 20t^5 + 20t^6 + t^8}{(1-t)^{20}}.$

\item $\widetilde{H}(R(3,5),t) = \displaystyle \frac{P(t)}{(1-t)^{29}},$ where the coefficients of $P(t)$ are 1, 6, 141, 931, 4816, 13916, 27531, 33391, 27531, 13916, 4816, 931, 141, 6, 1.

\item $\widetilde{H}(R(3,6),t) = \displaystyle \frac{P(t)}{(1-t)^{38}},$ where the coefficients of $P(t)$ are 1, 18, 626, 10246, 114901, 830484, 4081260, 13763184, 32507115, 54176230, 64224060, 54176230, 32507115, 13763184, 4081260, 830484, 114901, 10246, 626, 18, 1.

\item $\widetilde{H}(R(3,7),t) = \displaystyle \frac{P(t)}{(1-t)^{47}},$ where the coefficients of $P(t)$ are 1, 37, 2033, 62780, 1301634, 18067706, 173883458, 1186198090, 5851715254, 21192401230, 57013957462, 114926408114, 174616665986, 200665719450, 174616665986, 114926408114, 57013957462, 21192401230, 5851715254, 1186198090, 173883458, 18067706, 1301634, 62780, 2033, 37, 1.

\item $\widetilde{H}(R(4,3),t) = \displaystyle \frac{1 -3t + 9t^2 + 8t^3 + 9t^4 - 3t^5 + t^6}{(1-t)^{18}}.$

\item $\widetilde{H}(R(4,4),t) = \displaystyle \frac{P(t)}{(1-t)^{34}},$ where the coefficients of $P(t)$ are 1, 1, 141, 981, 8534, 39193, 139348, 325823, 556368, 652716, 556368, 325823, 139348, 39193, 8534, 981, 141, 1, 1.

\item $\widetilde{H}(R(5,3),t) = \displaystyle \frac{P(t)}{(1-t)^{27}}$, where the coefficients of $P(t)$ are 1, -6, 36, -70, 231, -189, 419, -189, 231, -70, 36, -6, 1.

\end{enumerate}

\subsection*{Acknowledgements}
I would like to thank my advisor Harm Derksen for suggesting the problem and his guidance throughout the project. I would also like to thank J.~Stembridge for useful conversations. For the computations in this paper, I used J.~Stembridge's SF package (see \cite{Stem}).

\end{document}